\documentclass[preprint,10pt,1p, sort&compress]{elsarticle}

\usepackage{graphicx}      

\usepackage{amsmath, amssymb, mathtools, amsfonts, amsthm}
\usepackage{bbding,pifont}  
	
\usepackage{csquotes}
\usepackage[normalem]{ulem} 
\usepackage[ngerman,english]{babel}
\usepackage{enumerate}
\usepackage{xparse} 

\usepackage{color}
\definecolor{TUMblau}		{RGB}{0, 101, 189}
\definecolor{TUMgruen}		{RGB}{162, 173, 0}
\definecolor{TUMorange}		{RGB}{227, 114, 34}
\definecolor{TUMdunkelrot}	{RGB}{156,013,022}
\definecolor{gray}			{RGB}{140,140,140}

\usepackage[hyperfootnotes=false,
pdffitwindow=false,						
pdftitle={Stability-Preserving, Adaptive Model Order Reduction of DAEs by Krylov-Subspace Methods},		
pdfauthor={Alessandro Castagnotto},		
linktocpage=true,						
colorlinks=true,						
linkcolor=TUMblau,						
citecolor=TUMdunkelrot,					
urlcolor=TUMgruen,						
draft = true
]{hyperref}

\hypersetup{
	pdfinfo={
		Keywords={Model reduction; Differential-algebraic equation; Stability preservation}
	}
}

\usepackage{booktabs,units,pmat}

\usepackage{algorithm, algpseudocode}

	\newlength\mywidth
\newcommand{\Reals}{\mathbb{R}}
\newcommand{\Complex}{\mathbb{C}}
\newcommand{\Htwo}{\mathcal{H}_2}
\renewcommand{\th}{^{\text{th}}}

\newcommand{\defeq}{\vcentcolon=}
\newcommand{\logicand}{\wedge}

\newcommand{\trans}[1]{#1^\top}

\newcommand{\inv}[1]{#1^{-1}}
\newcommand{\invt}[1]{#1^{-\top}}
\newcommand{\pd}{\succ}
\newcommand{\nd}{\prec}

	\newcommand\vecspace[1]{\mathcal{#1}}
\DeclareMathOperator{\diagmat}{diag}
	\newcommand\diag[1]{\diagmat\left(#1\right)}
\DeclareMathOperator{\matrixrank}{rank}
	\newcommand\rank[1]{\matrixrank #1}

\newcommand{\norm}[1]{\left\lVert#1\right\rVert}
\newcommand{\diff}[2]{\frac{\mathrm{d}\, #1}{\mathrm{d}\, #2}}
\newcommand{\innProd}[2]{\left\langle#1, \; #2\right\rangle}

\newcommand{\ts}{\!} 

\newcommand{\krylov}[3]{\vecspace{K}_{#1}\left(#2, #3\right)}

\newtheorem{thm}{Theorem}[section]

\newtheorem{proposition}{Proposition}
\newtheorem{lemma}{Lemma}
\newtheorem{corollary}{Corollary}
\newtheorem{definition}{Definition}

\journal{arXiv}

\begin{document}
\begin{frontmatter}

\title{Stability-Preserving, Adaptive Model Order Reduction of DAEs by Krylov-Subspace Methods\tnoteref{t1}}
\tnotetext[t1]{The work related to this contribution is supported by the German Research Foundation (DFG), Grant LO408/19-1.}

\author[RT]{Alessandro Castagnotto \corref{corr1}}\cortext[corr1]{Corresponding author} \ead{a.castagnotto@tum.de}
\author[RT]{Heiko K. F. Panzer}
\author[Mathe]{Klaus-Dieter Reinsch}
\ead{kladire@ma.tum.de}
\author[RT]{Boris Lohmann}


\address[RT]{Chair of Automatic Control, Boltzmannstr. 15}
\address[Mathe]{Chair of Numerical Mathematics, Boltzmannstr. 3}
\address{Technische Universit\"at M\"unchen, D-85748 Garchig, Germany}

\begin{abstract}
Systems of differential-algebraic equations (DAEs) represent a widespread formalism in the modeling of constrained mechanical systems and electrical networks. Due to the automatic, object-oriented generation of the equations of motion and the resulting redundancies in the descriptor variables, DAE systems often reach a very high order. This motivates the use of model order reduction (MOR) techniques that capture the relevant input-output dynamics in a reduced model of much smaller order, while satisfying the constraints and preserving fundamental properties. 
Due to their particular structure, new MOR techniques designed to work directly on the DAE are required that reduce the dynamical part while preserving the algebraic. 
In this contribution, we exploit the specific structure of index-1 systems in semi-explicit form and present two different methods for stability-preserving MOR of DAEs. The first technique preserves strictly dissipativity of the underlying dynamics, the second takes advantage of $\Htwo$-pseudo-optimal reduction and further allows for an adaptive selection of reduction parameters such as reduced order and Krylov shifts.
\end{abstract}


\end{frontmatter}
\newcommand{\Gprop}{G_f(s)}
\newcommand{\Gimp}{G_\infty(s)}
\newcommand{\Pencil}{P_{A,E}(\lambda)}
\newcommand{\Dr}{D_{imp}}
\newcommand{\rtan}{r}
\newcommand{\ltan}{l}
\newcommand{\Rtan}{R}
\newcommand{\Ltan}{L}
\DeclareDocumentCommand{\Grt}{o o}{%
	{ %
		\IfNoValueTF {#1} %
		{\widetilde{G}_{r}(s)}%
		{ %
			\IfNoValueTF{#2} %
			{\widetilde{G}_{r,#1}(s)} %
			{\widetilde{G}_{r,#1}^{#2}(s)} %
		}
	}%
}

\section{Introduction}\label{sec:Intro}
The accurate modeling of dynamical systems such as flexible mechanical structures, electrical circuits with large-scale integration and interconnected power networks often results in mathematical models of very high dimension, meaning that the number of state variables and equations used to represent the dynamic behavior grows to the order of tens of thousands up to millions. 

This \emph{curse of dimensionality} is even more accentuated if the formalism used to model the system yields a set of \emph{differential-algebraic equations} (DAEs). 
This type of formulation arises e.g. if the system is modeled by equating the fundamental physics of each individual element, which is then linked to its neighbors through appropriate interconnections. This approach has the fundamental advantage of being quite generic, making it therefore suited for automated, object-oriented computerized modeling. 
DAEs represent systems of equations that are composed of a dynamical part, describing how the state can vary in time, and an algebraic part, describing static relationships and constraints amongst state variables. 
Subsequently, within this formalism the states used to model the system do not represent minimal coordinates needed to describe the dynamics but generally include a set of algebraic variables used for expressing a \emph{constraint} manifold on which the state of the system has to evolve at all times.
Due to this additional algebraic variables, the number of the states used to model the system becomes even larger.

Especially in a large-scale setting, it is not computationally easy---if at all possible---to separate the dynamic states from the algebraic ones. This implies that the DAE has to be analyzed and treated directly, motivating the flourishing theory on DAEs that has been developed in the past decades \cite{Dai_1989, Kunkel_2006, Schoeps_2014}.

In the following contribution, we shall focus on the simpler, yet common case of \emph{linear, constant coefficients DAEs} given as a generalized state-space system of the form
\begin{equation}
\begin{aligned}
E \,\dot{x} &= A\, x + B\, u \\
y &= C\, x + D\, u
\end{aligned}
\label{eq:DAE}
\end{equation}
where $E\ts\in\ts\Reals^{N \ts\times\ts N}$ is the singular \emph{descriptor matrix}, $A\ts\in\ts\Reals^{N\ts\times\ts N}$ is the system matrix and $x\ts\in\ts\Reals^N$, $u\ts\in\ts\Reals^m$, $y\ts\in\ts\Reals^p$ ($p,m\ts\ll\ts N$) represent the state, input and output of the system respectively. Whenever the matrix $E$ is regular, we will refer to \eqref{eq:DAE} as a system of ordinary differential equations (ODE) in implicit form.

If the model in \eqref{eq:DAE} is used for simulation, optimization or control synthesis, a high order $N$ might pose severe limitations in terms of computation time and even storage capabilities.
For this reason, model order reduction (MOR) techniques are implemented to obtain a significantly smaller reduced order model (ROM) that captures the dominant input-output behavior of the full order model (FOM) while preserving characteristic properties.

A common approach to the reduction of state-space models like in \eqref{eq:DAE} is given by the \emph{Petrov-Galerkin projection} $\Pi^2  = \Pi =  E\,V\inv{\left(\trans{W}E\,V\right)}\trans{W}$ and yields
\begin{equation}
\begin{aligned}
\overbrace{\trans{W}E\,V}^{E_r} \, \dot{x}_r\, &= \, \overbrace{\trans{W}A\,V}^{A_r}\, x_r \, +  \, \overbrace{\trans{W} B}^{B_r} \, u\\ 
y_r \, &= \; \underbrace{C\,V}_{C_r}\, x_r \,+ \; \underbrace{D}_{D_r} \, u
\end{aligned}\label{eq:ROM}		 
\end{equation}
where $x_r\ts\in\ts\Reals^{n}$ $(n\ts\ll\ts N)$ represents the reduced state vector. 
Accordingly, the aim of MOR techniques lies entirely in the appropriate choice and design of projection matrices $V, W$.

MOR theory for linear state-space systems satisfying $\det E\ts\neq\ts0$ has been extensively studied over the past decades and includes established techniques such as \emph{modal reduction}, \emph{balanced truncation} and \emph{moment matching} \cite{Antoulas_Book}. The latter, also known as \emph{Krylov-subspace methods}, \emph{implicit moment matching} or \emph{interpolatory reduction} techniques, stand out due to their generality and low computational cost,  making them a predestined candidate for the reduction of truly large-scale systems.
However, these reduced computations reveal only local information about the system and global properties of the model might be lost.
For instance, one challenge arising when applying moment matching is that, given a stable FOM, stability of the ROM is not guaranteed per se. 
While there are some approaches to preserve stability in the case of ODEs, stability preservation in the more general case of DAEs has---to the authors' knowledge---not been addressed so far.

In this contribution this problem is considered. Two different techniques for stability preservation in interpolatory reduction are presented. Both techniques apply to index-1 DAEs in semi-explicit form and do not require the explicit computation of the underlying ordinary differential equation. Further it will be shown in theory and through numerical examples how an orthogonal projection ($V\ts=\ts W$) of such DAEs can be applied correctly only if some conditions are met. Finally, we will explain how to adaptively choose reduced order and interpolation points of the Krylov subspaces for this class of systems.

The sequel of this contribution is outlined as follows: section \ref{sec:Preliminaries} revises the existing theory on model reduction, stability preservation and DAEs. 
In section \ref{sec:StabPrev1-orthogonal} a first stability preserving reduction technique is presented. This requires the extension of the concept of \emph{strictly dissipative} realization of a state-space system to the more general case of DAEs and shows how to exploit this property for stability preservation. The correct reduction of DAEs by orthogonal projection will also be discussed.
Section \ref{sec:Stability_Preserving_H2_pseudo} presents the second stability preserving technique based on $\Htwo$-pseudo-optimal reduction. For this purposes, the \emph{Pseudo-Optimal Rational Krylov} (PORK) algorithm will be adapted. Further, the adaptive choice of reduced order and interpolation points will also be generalized.
Finally, the theoretical results are validated through numerical examples in section \ref{sec:examples}.

\section{Preliminaries and problem statement}\label{sec:Preliminaries}
\subsection{Model reduction of ODEs by Krylov-subspace methods}
As introduced in section \ref{sec:Intro}, the computation of a reduced model by means of projection boils down to finding suitable projection matrices $V, W$ (cf. equation \eqref{eq:ROM}). Depending on which properties of the original model we would like to preserve, several methods of determining the projection subspace have been studied in the literature. 
For instance, \emph{modal reduction} is aimed at preserving dominant eigenmodes of the system \cite{Bonvin_1982_Modal}, while \emph{balanced truncation} preserves dominant directions in the system that are equally controllable and observable \cite{Antoulas_Book}. 
Due to their generality and low computational cost, \emph{Krylov-subspace} methods have been studied intensively over the past decades \cite{Grimme_PhD,Gallivan_2004_MIMO,Beattie_2014_Survey} and shall be introduced briefly in the following.

If the matrices $V,W$ in \eqref{eq:ROM} span $n\th$-order \emph{tangential input} and/or \emph{output Krylov subspaces}%
\begin{subequations}%
	\begin{align} %
	&\krylov{n}{\inv{(A-s_0E)}E}{\inv{(A-s_0E)}B\;\rtan} \label{eq:InputKrylov}\\
	&\krylov{n}{\invt{(A-s_0E)}\trans{E}}{\invt{(A-s_0E)}\trans{C}\;\ltan}\label{eq:OutputKrylov}
	\end{align}\label{eq:KrylovSubspaces}%
\end{subequations}%
respectively for some complex \emph{shifts} $s_0\in \Complex$ and \emph{tangential directions} $\rtan\in \Reals^{m}$, $\ltan\in \Reals^{p}$, then the reduced model matches some coefficients of the Taylor series expansion of the original transfer function around $s_0$.
In the case of systems with more than one input or output ($m,p\ts>\ts 1$) the Taylor coefficients are not scalar and the matching applies only to the linear combination defined by the tangential directions $\rtan$ and $\ltan$.
The computation of such matrices requires only the solution of (typically sparse) linear systems of equations, which makes it particularly attractive in a large-scale setting. In this framework, the task of computing good reduced models translates to finding suitable shifts and tangential directions.

For theoretical considerations it is worth noting that, at least in the ODE case, there is a duality between the Krylov subspaces in \eqref{eq:KrylovSubspaces} and the \emph{sparse-dense Sylvester equations}%
\begin{subequations}%
	\begin{align}
	&A\,V - E\,V S_V - B \,\Rtan = 0 \label{eq:V-Sylvester}\\
	&\trans{W}A -S_W \,\trans{W}E -\Ltan\,C = 0 \label{eq:W-Sylvester}
	\end{align}\label{eq:Sylvester}%
\end{subequations}%
respectively, where the pairs $(S_V, \Rtan)$ and $(S_W, \Ltan)$ encode the interpolation data (shifts and tangential directions) \cite{Gallivan_2004_Sylvester,Wolf_PhD}.

While reduction methods like modal techniques and balanced truncation intrinsically preserve stability, this is generally not the case for Krylov-subspace methods. This can be justified by the fact that while the former methods require costly computations that reveal the structure of the system (eigenvectors or Gramians respectively), the latter method limits the effort to the computation of only a few directions that guarantee matching of the transfer behavior at certain frequencies.
It is therefore highly desirable to include some strategy to ensure stability of the ROM even for interpolatory model reduction nonetheless.

\subsection{Stability preservation in MOR of ODEs}\label{sec:ODE_Stability}
A few results are available for the stability preservation in interpolatory reduction of ODEs and shall be revised in the following.

\subsubsection{Strictly-dissipative realizations}\label{sec:SD}
A state-space system as in \eqref{eq:DAE} satisfying $\det\ts E\ts\neq\ts 0$ is said to be in a \emph{strictly dissipative} form if and only if $E\ts=\ts\trans{E}\ts\pd\ts0$ and $A\ts+\ts\trans{A}\ts\nd\ts0$. Since the trajectory of such a system can be shown to be  strictly decreasing with respect to the elliptic 2-norm induced by $E$ \cite{Panzer_PhD}, i.e.
\[
\norm{x(t)}_{2,E}^2 \defeq \trans{x}(t)\, E\, x(t),
\]
systems in strictly dissipative form are asymptotically stable. 
In other terms, $E\ts=\ts\trans{E}\ts\pd\ts0$ and $A\ts+\ts\trans{A}\ts\nd\ts0$ implies that all generalized eigenvalues of the pair ($A$, $E$) have strictly negative real part.
This fact can be exploited in a projective MOR setting as in \eqref{eq:ROM} by performing orthogonal projection with a full-column-ranked matrix $V\ts=\ts W$. This type of (Galerkin) projection preserves definiteness of the original matrices and it is therefore easy to show that the reduced model in \eqref{eq:ROM} will preserve strict dissipativity and hence be stable \cite{Silveira_1999}.
Please note that any asymptotically stable ODE can be transformed to a strictly dissipative realization and that systems in \emph{second order} structure can often be directly transformed into an equivalent strictly dissipative first order realization \cite{Panzer_PhD}.

\subsubsection{$\Htwo$-pseudo-optimality and CUREd SPARK}\label{sec:H2-opt}
Recently, interpolatory techniques for ODEs that adaptively choose the Krylov shifts and guarantee stability preservation at convergence have been proposed \cite{Beattie_2009_TrustRegion, Panzer_2013_ACC, Panzer_PhD}.
In this contribution, the focus will lie on the \emph{Stability-Preserving, Adaptive Rational Krylov} (SPARK) Algorithm introduced by Panzer et al. in \cite{Panzer_2013_ACC} and further developed in \cite{Panzer_PhD} due to its suitability in the reduction of DAEs.
The stability preservation in SPARK is guaranteed by construction, exploiting the concept of $\Htwo$-pseudo-optimality \cite{Wolf_2013_ECC, Wolf_PhD}. Given the desired interpolation data encoded in $(S_V, \Rtan)$ (cf. equation \eqref{eq:Sylvester}), it is possible to directly construct a reduced model which satisfies the interpolatory conditions and shares the same spectrum as $(-S_V)$. 
Based on the Sylvester equation \eqref{eq:Sylvester}, the Pseudo-Optimal Rational Krylov (PORK) algorithm yields the desired pseudo-optimal reduced model \cite{Wolf_2013_ECC}.

\begin{algorithm}[!ht]\caption{Pseudo-Optimal Rational Krylov (PORK)}\label{alg:pork}
	\begin{algorithmic}[1]
		\Require ($E$, $A$, $B$, $C$, $D$), $(S_V, \Rtan)$
		\Ensure $\mathcal{H}_2$ pseudo-optimal reduced system matrices
		\State {$V \gets \ A V - E V S_V - B \Rtan = 0$}
		\hfill{// Krylov subspace}
		\State {$\inv{P_r} =$ lyap$(-\trans{S_V}, \trans{\Rtan}\Rtan)$} 	\hfill{// Low-dimensional Lyapunov eq.}
		\State {$B_r = -P_r \trans{\Rtan}$}
		\State {$A_r = S_V + B_r \Rtan$, $\;E_r = I$, $\;C_r = CV$, $\;D_r = D$}
	\end{algorithmic}
\end{algorithm}

A dual version of the algorithm based on the output Krylov Sylvester equation \eqref{eq:W-Sylvester} is available and analogous.
The reason why such a reduced model is called  $\Htwo$-pseudo-optimal is that it is the global optimum, in terms of the $\Htwo$ norm, amongst all reduced models with the same spectrum.
Therefore, the PORK algorithm assigns the eigenvalues of the reduced model, resulting from a projection with a Krylov subspace, according to the choice of shifts. It is therefore evident that selecting shifts on the right complex half-plane yields a stable, pseudo-optimal ROM by construction.
However, note that in this setting, the choice of appropriate shifts becomes \emph{twice} as important: not only do they determine at which complex frequencies interpolation is achieved, but they also explicitly determine the eigenvalues of the reduced model.
For this reason, pseudo-optimal reduction has been introduced concurrently to an adaptive reduction framework by Panzer et al. in \cite{Panzer_2013_ACC}, in which the choice of shifts for pseudo-optimal reduction is conducted adaptively---through a greedy algorithm---within a cumulative reduction framework. A complete and thorough exposition of the proposed reduction strategy, known as \emph{CUREd SPARK}, can be found in the recent dissertations  of Panzer \cite{Panzer_PhD} and Wolf \cite{Wolf_PhD} and shall be revised briefly in the following.

Cumulative reduction (CURE) is based on a factorization of the error system that allows an iterative reduction of the error term and therefore an adaptive choice of reduced order. In fact, if the reduced model, denoted by its transfer function $G_r(s)$, results from a projection with an input Krylov subspace $V$, then the error can be factorized as
\begin{equation}
G_e(s) =
\underbrace{
	\left[\begin{array}{c|c}
	E,A & B\\ \hline C & D
	\end{array}\right]
}_{G(s)}
-
\underbrace{
	\left[\begin{array}{c|c}
	E_r,A_r & B_r\\ \hline C_r & D
	\end{array}\right]
}_{G_r(s)} 
= 
\underbrace{
	\left[\begin{array}{c|c}
	E,A & B_\bot\\ \hline C & 0
	\end{array}\right]
}_{G_\bot(s)}
\cdot
\underbrace{
	\left[\begin{array}{c|c}
	E_r,A_r & B_r\\ \hline \Rtan & I
	\end{array}\right]
}_{\Grt}.
\label{eq:ErrorFactorization}
\end{equation} %
The first factor $G_\bot(s)$ represents a high-dimensional system that resembles the original one but is excited only by the input directions not considered during the previous projection by $\Pi$: $B_\bot \defeq (I-\Pi)B = B - EV\inv{E_r}B_r$.
The second factor $\Grt$ is low-dimensional and resembles the reduced model except for having a unity feedthrough and the output matrix corresponding to the matrix of right tangential directions $\Rtan$.

Using this error factorization, a cumulative reduction procedure is derived which reduces, at each iteration $k$, only the high-dimensional part  $G_{\bot,k}$ of the new error and cumulates all ROMs into a total reduced model $G_{r,k}^{\Sigma}$:
\begin{equation} \label{eq:CURE}
\begin{aligned}
G(s) &= G_{r,1}(s) + \underbrace{\left(G_{\bot,1}(s)\cdot \Grt[1] \right)}_{G_{e,1}(s)}\\
&= G_{r,1}(s) + \left(G_{r,2}(s) +G_{\bot,2}\cdot\Grt[2]\right)\cdot \Grt[1]\\
&\mathrel{\makebox[\widthof{=}]{\vdots}}\\
&= G_{r,k}^{\Sigma}(s) + G_{\bot,k}(s)\cdot \Grt[k][\Sigma].
\end{aligned}
\end{equation}
A detailed description of the procedure including all expressions, the computation of the cumulated system matrices as well as the dual factorization using the output Krylov subspace $W$	are given in \cite[pp.57-70]{Panzer_PhD}.

Within CURE, there is no restriction of the reduced order at each iteration $k$. One possible choice is to reduce the $G_{\bot,k}(s)$ system at each iteration to a reduced model $G_{r,k}(s)$ of order two. This idea is exploited by the Stability-Preserving Adaptive Rational Krylov (SPARK) algorithm, which is aimed at finding a local $\Htwo$-optimal reduced model of order two while reducing the search space to the set of all $\Htwo$-pseudo-optimal, stable ROMs. Note that this restriction is valid since $\Htwo$-pseudo-optimality is a necessary condition for $\Htwo$-optimality, hence the set of all $\Htwo$-optimal ROMs is entirely contained in the set of all $\Htwo$-pseudo-optimal ROMs. This restriction of the search space has---amongst other---the advantage of simplifying the cost function for optimization. In fact, if $G_r(s)$ is an $\Htwo$-pseudo-optimal interpolant of $G(s)$, then the error norm simplifies to \cite[p.92]{Wolf_PhD}:
\begin{equation} \label{eq:H2errorPO}
\begin{aligned}
\norm{G_e(s)}_{\Htwo}^2 &\mathrel{\makebox[\widthof{$\overset{\Htwo-po}{=}$}]{=}} \innProd{G_e(s)}{G_e(s)}_{\Htwo}\\
&\mathrel{\makebox[\widthof{$\overset{\Htwo-po}{=}$}]{=}} \norm{G(s)}_{\Htwo}^2 - 2\innProd{G(s)}{G_r(s)}_{\Htwo} + \norm{G_r(s)}_{\Htwo}^2 \\
&\overset{\Htwo-po}{=} \norm{G(s)}_{\Htwo}^2 - \norm{G_r(s)}_{\Htwo}^2.
\end{aligned}
\end{equation}
Since the first term in the resulting expression is constant, the minimization of the error norm is equivalent to the maximization of the norm of the reduced model, a cost function that is cheap to compute! Restricting the search to $\Htwo$-pseudo-optimal, stable reduced order models of order two yields simple expressions for gradient and Hessian, parametrized by two real numbers. Optimal values can be found by optimization, for example through a trust-region descent method \cite[pp.75-82]{Panzer_PhD}.
Cumulation of all \mbox{$\Htwo$-optimal} reduced models of order two yields the final ROM of the desired order.

\subsection{Differential-algebraic equations}\label{sec:DAE}
This section revises the main properties of linear, constant coefficient DAEs and introduces the special case of index-1 DAEs in semi-explicit form (SE-DAEs).


A DAE as in \eqref{eq:DAE} is primarily characterized by the matrix pair $(A,E)$ defining the \emph{matrix pencil} $\Pencil \ts\defeq\ts A \ts+\ts\lambda E$. Provided that the pair $(A,E)$ has a finite number of generalized eigenvalues, i.e. $\det \Pencil\ts\not\equiv\ts 0$, the pencil is called regular, which is often the case if the algebraic constraints are not redundant. In this case the generalized eigenvalues can be divided into those that take a finite value and those at infinity. In terms of the dynamic system in \eqref{eq:DAE}, the former correspond to the dynamic modes whereas the latter correspond to the kernel of $E$, hence the algebraic part. Regularity of the pencil also implies that the DAE has a unique solution for a sufficiently smooth input and admissible initial conditions, and that there exist two invertible matrices $Q_l,Q_r$ that bring the system to the so called \emph{Kronecker-Weierstra\ss} canonical form \cite{Kunkel_2006}:
\begin{equation}
\begin{aligned}
\overbrace{\left[
	\begin{array}{cc}
	I_{n_f} & 0\\
	0 & N
	\end{array}
	\right]}^{Q_l\,E\,Q_r}
\left[\begin{array}{c}
\dot{x}_f\\
\dot{x}_\infty
\end{array}\right]
&= 
\overbrace{\left[
	\begin{array}{cc}
	J & 0\\
	0 & I_{n_\infty}
	\end{array}
	\right]}^{Q_l\,A\,Q_r}
\left[\begin{array}{c}
x_f\\
x_\infty
\end{array}\right]
+ 
\overbrace{\left[\begin{array}{c}
	B_f\\
	B_\infty
	\end{array}\right]}^{Q_l\,B}
u \\
y &= \underbrace{\left[C_f, C_\infty \right]}_{C\,Q_r}
\left[\begin{array}{c}
x_f\\
x_\infty
\end{array}\right] + D\, u.
\end{aligned}
\label{eq:KWCF}
\end{equation}
A regular DAE can thereby be decoupled in a dynamical part, represented by $x_f$, having finite eigenvalues according to the Jordan matrix $J$, and an algebraic part $x_\infty$ with all eigenvalues at infinity, described by $N$, a nilpotent matrix of nilpotency index $\nu$. This index induces a predominant characterization of DAEs in terms of its ``distance'' from an ODE and is often used to define the type of DAE at hand. Analogously to the case of ODEs, asymptotic stability of the DAE can be analyzed by inspecting the finite eigenvalues.
From the canonical form \eqref{eq:KWCF}, the transfer function of the DAE can be easily computed as \cite{Dai_1989}
\begin{equation}
G(s) = \underbrace{C_f \inv{\left(sI_{n_f}-J\right)} B_f + D}_{\Gprop}
+ \underbrace{C_\infty \left(-\sum_{i=0}^{\nu-1}N^is^i \right)B_\infty}_{\Gimp}
\label{eq:DAE_tf}
\end{equation}
It is a characteristic property of DAEs that the transfer function is generally composed of a \emph{proper} part $\Gprop$, resulting from the underlying ODE, and an \emph{improper} part $\Gimp$, a polynomial in $s$ of degree $\nu-1$.
In terms of reduction this has a very important implication: whenever the DAE presents improper behavior, the improper part has to be matched exactly to avoid unbounded errors at high frequencies. 
Generally speaking, the correct reduction of a DAE requires the decomposition of the system into dynamic and algebraic part,
in order to satisfy the algebraic constraints exactly and approximate the differential part. However, the computation of the canonical form \eqref{eq:KWCF} or, equivalently, of projectors onto the respective \emph{deflating subspaces}, is numerically ill-conditioned and not feasible in the large-scale setting. Fortunately, for some DAEs with particular structures it is possible to identify analytically the relevant subspaces without computing the canonical form.

\subsubsection{Semi-explicit, index-1 DAEs}
The special structure of DAE we will consider in this contribution is that of semi-explicit index-1 DAEs (henceforth SE-DAEs), i.e. generalized state-space systems of the form
\begin{equation}
\begin{aligned}
\left[
\begin{array}{cc}
E_{11} & 0\\
0 & 0
\end{array}
\right]
\left[\begin{array}{c}
\dot{x}_1\\
\dot{x}_2
\end{array}\right]
&= 
\left[
\begin{array}{cc}
A_{11} & A_{12}\\
A_{21} & A_{22}
\end{array}
\right]
\left[\begin{array}{c}
x_1\\
x_2
\end{array}\right]
+ 
\left[\begin{array}{c}
B_{11}\\
B_{22}
\end{array}\right]
u \\
y &= \left[C_{11}, C_{22}\right] 
\left[\begin{array}{c}
x_1\\
x_2
\end{array}\right] + D\, u
\end{aligned}
\label{eq:SE}
\end{equation}
where $\det\ts A_{22}, \det\ts E_{11}\ts\neq\ts 0$, $E_{11}\ts\in\ts\Reals^{n_{dyn}\times n_{dyn}}$ and all other matrices partitioned accordingly. $n_{dyn}\ts=\ts\rank{E_{11}}$ is called the \emph{dynamic order} of the DAE.
These systems arise frequently in electrical networks and power systems \cite{Freitas_2008,Rommes_2006,MORWiki_power}. If the semi-explicit form is not directly given it can often be achieved by a simple reordering of rows and columns.
The reason why a DAE as in \eqref{eq:SE} is called semi-explicit is that the \emph{underlying ODE} can be obtained replacing $x_2$ through the algebraic equation, which yields:
\begin{equation}
\begin{aligned}
\overbrace{E_{11}}^{E_1} \,\dot{x}_1 &= \overbrace{A_{11}-A_{12}\inv{A_{22}}A_{21}}^{A_1}\, x_1 + \overbrace{B_{11} - A_{12}\inv{A_{22}} B_{22}}^{B_1}\, u \\
y &= \underbrace{C_{11} - C_{22}\inv{A_{22}}A_{21}}_{C_1}\, x_1 + \underbrace{D - C_{22}\inv{A_{22}}B_{22}}_{D_1}\, u.
\end{aligned}
\label{eq:underlying ODE}
\end{equation}
While the underlying ODE is useful for theoretical considerations, in a large-scale setting it is not desirable, if at all feasible, to compute it explicitly. Note that in many applications, the dimension of the algebraic state-space $x_2$ is significantly larger than $x_1$. The reduction has therefore to be conducted based on the implicit representation \eqref{eq:SE}.	

\subsection{DAE-aware model reduction by Krylov-subspace methods}

In general, the reduction of DAEs cannot be conducted by applying standard methods for ODEs. The reason is that while in the ODE case the whole system can be reduced, for the DAE it is important to reduce only the dynamic part and keep the algebraic part responsible for the improper behavior seen in \eqref{eq:DAE_tf}.
Therefore, the general approach in the reduction of DAEs is to first identify and separate the dynamic from the algebraic part and subsequently reduce only the former while preserving the latter. Generally speaking, this requires the computation of deflating subspaces and the solution of either projected generalized Lyapunov equations (cf. \cite{Stykel_2004}) or the computation of projected Krylov subspaces (cf. \cite{GugercinStykel_2013}) depending on the choice of reduction method. 
A different approach that separates algebraic and dynamic part using the concept of \emph{tractability index} is discussed in \cite{Ali_2013} and generalized in \cite{Banagaaya_2014}.

This treatise focuses on Krylov-subspace methods due to their generality and computational advantages.
If the DAE presents a special structure, it is generally convenient to avoid the explicit computation of the respective subspaces and directly derive algorithms that implicitly reduce the ODE. This is done for example by Ahmad  and Benner in \cite{Ahmad_2014} for special index-3 systems and in \cite{Ahmad_2015} for special bilinear systems, as well as by Gugercin et al. in \cite{GugercinStykel_2013} for special index-1 and index-2 systems.
Since we will be focusing on SE-DAEs, we shall revise in the following the DAE-aware reduction strategy described in \cite{GugercinStykel_2013}. This approach is similar to the standard ODE two-sided Rational Krylov and adds appropriate shifting terms to ensure interpolation of the underlying ODE and matching of the improper part. 
\begin{thm}[Adapted from \cite{GugercinStykel_2013}]
	Let the full order model be a SE-DAE as in \eqref{eq:SE}. Then, the improper part $\Gimp$ of its transfer function defined in \eqref{eq:DAE_tf} is given by
	\begin{equation}
	\Gimp = \Dr \defeq - C_{22}\inv{A_{22}}B_{22}.
	\label{eq:P(s)}
	\end{equation}
	Further, the reduction by tangential interpolation of the underlying ODE can be achieved implicitly by computing the reduced model according to
	\begin{equation}
	\begin{aligned}
	E_r &= \trans{W}E\,V, \quad A_r = \trans{W}A\,V + \Ltan\,\Dr\,\Rtan\\
	B_r &= \trans{W}B + \Ltan\,\Dr, \quad C_r = C,V + \Dr\,\Rtan \\
	D_r &= D + \Dr
	\end{aligned}	
	\end{equation}
	where $V$ and $W$ are tangential input and output Krylov subspaces computed with the tangential directions in $\Rtan$ and $\Ltan$ respectively.
	\label{thm:Gugercin}
\end{thm}
\noindent
Note that since the DAE considered is of index $\nu\ts=\ts1$, the transfer function cannot be improper but can at most contain a constant implicit feedthrough. We shall denote this term $\Dr$ to underline that it is the feedthrough \emph{implicitly} retained in the DAE, as opposed to the explicit feedthrough term D.

We shall revisit this theorem in section \ref{sec:StabPrev1-orthogonal}, where we will give a new proof based on a special Sylvester equation. We will also underline the relationship between the implicit DAE reduction and the explicit reduction of the underlying ODE, as this is of central importance for the fidelity of the ROM. In this context, it will be shown that the reduction proposed in theorem \ref{thm:Gugercin} cannot be applied for orthogonal projections $V$=$W$ of arbitrary SE-DAEs. 

\subsection{Problem statement}
The goal of this contribution is to introduce techniques that can be used to ensure the stability of the ROM obtained by Krylov-based reduction of SE-DAEs. This is done by extending the concepts and methods introduced in this section to this class of DAEs. 
Even though the structure of SE-DAEs suggests how to compute the underlying ODE explicitly, the computations required for the DAE-aware, stability preserving reduction will use the original DAE matrices only. It will be shown how the proposed procedures are equivalent to a the corresponding reduction performed directly on the underlying ODE. Numerical investigations on different benchmark systems will be used to assess the effectiveness of the proposed methods.

\section{Stability-preserving reduction for strictly dissipative DAEs}\label{sec:StabPrev1-orthogonal}
The first case of stability preserving reduction is suitable for SE-DAEs that are given in a strictly dissipative realization. Similarly to the ODE case, orthogonal projection will be used to preserve dissipativity. Nevertheless, note that orthogonal projection cannot be conducted correctly on any SE-DAE, even by applying the DAE-aware MOR scheme of theorem \ref{thm:Gugercin}. In general, the resulting ROM will fail to approximate only the dynamic part and lose dissipativity. The conditions the DAE has to satisfy in order to obtain a correct reduction in this sense will be derived in the following.

\subsection{Strictly dissipative SE-DAEs}
The concept of a strictly dissipative formulation introduced in section \ref{sec:SD} for ODEs can be naturally extended to SE-DAEs by inspection of the underlying ODE, as stated in the following definition
\begin{definition}\label{def:SD_DAE}
	A SE-DAE System as in \eqref{eq:SE} with underlying ODE \eqref{eq:underlying ODE} is said to be in \emph{strictly dissipative form} if and only if 
	\[
	E_1 = \trans{E_1} \pd 0 \quad \logicand \quad A_1+\trans{A_1} \nd 0.
	\]
\end{definition}
\noindent
Clearly this definition respects the same properties as the one in section \ref{sec:SD} and, in particular, it implies stability of the SE-DAE.
Note that in some cases strict dissipativity of the SE-DAE can be assessed without explicitly computing the underlying ODE, as it 	is explained in the next proposition.
\begin{proposition}\label{PROP:REINSCH}
	Let $A$ and $A_1$ be defined as in \eqref{eq:SE} and \eqref{eq:underlying ODE} respectively. Assume $A +\trans{A} \nd 	0$. Then $ A_1+\trans{A_1} \nd 0$.
\end{proposition}
\begin{proof}
	The proof is given in \ref{appx:Reinsch}.
\end{proof}
\noindent
Obviously, not all SE-DAEs are given in strictly dissipative form. Nonetheless, if the SE-DAE is asymptotically stable, then 	analogously to the ODE case, a transformation to strictly dissipative form can be found:

\begingroup
\renewcommand*{\arraystretch}{1.5}
\renewcommand*{\arraycolsep}{5pt}
\begin{lemma}\label{thm:LMI}
	Assume the SE-DAE in \eqref{eq:SE} is asymptotically stable. A strictly dissipative formulation is obtained by multiplying the DAE from the left with the matrix
	\[
	T = \left[\begin{array}{c c}
	\trans{E_{11}}P & -\trans{E_{11}}PA_{12}\inv{A_{22}}\\
	0 				&	I
	\end{array}	\right],
	\]
	where $P\ts=\trans{P}\ts\pd\ts0$ solves the Lyapunov inequality
	\[
	\trans{E_{11}}PA_1 + \trans{A_1}PE_{11} \nd 0
	\]
	and $A_1$ is the Schur complement defined in \eqref{eq:underlying ODE}.
\end{lemma}
\begin{proof}
	Note that the transformation $T$ can be factorized as
	\[
	T = \left[\begin{array}{cc}
	\trans{E_{11}}P & 0\\
	0 				&	I
	\end{array}	\right]\cdot
	\left[\begin{array}{cc}
	I & -A_{12}\inv{A_{22}}\\
	0 				&	I
	\end{array}	\right].
	\]
	Applying this transformation to the SE-DAE yields
	\[
	\left[
	\begin{array}{cc}
	\trans{E_{11}}PE_{11} & 0\\
	0 & 0
	\end{array}
	\right]
	\dot{x}
	= 
	\left[
	\begin{array}{cc}
	\trans{E_{11}}PA_{1} & 0\\
	A_{21} & A_{22}
	\end{array}
	\right]
	x
	+ 
	\left[\begin{array}{c}
	\trans{E_{11}}PB_{1}\\
	B_{22}
	\end{array}\right]
	u
	\]
	which is in strictly dissipative form by construction.
\end{proof}
\endgroup
\noindent
Therefore, finding a strictly dissipative transformation reduces to solving a convex feasibility problem, which might still be difficult in the large-scale case. Alternatively, a Lyapunov equation with appropriate right-hand side might be solved.
Note at this point that if the FOM is not given in strictly dissipative form, stability preserving reduction can still be achieved by  using the second method proposed in section \ref{sec:Stability_Preserving_H2_pseudo}.

\subsection{Orthogonal projection of SE-DAEs}\label{sec:OrthogonalProjectionOfDAEs}
Combining what has been discussed so far, we are now ready to establish the model reduction strategy that acts only on the sparse matrices of the SE-DAE and ensures stability of the ROM.
Similarly to the ODE case, we will approach the problem by applying orthogonal projection to the DAE. 
First, we start by demonstrating how the reduction framework as in theorem \ref{thm:Gugercin} generally fails to comprise the special case of orthogonal projection. This is counterintuitive, since for ODEs orthogonal projection is a special case of skew projection ($V\ts\neq\ts W$).
To do so, we compare from a theoretical standpoint the reduction of the underlying ODE to the implicit reduction of the DAE.

Recalling the SE-DAE \eqref{eq:SE} and the corresponding explicit representation of the underlying ODE \eqref{eq:underlying ODE} we start by showing the equivalence of the Krylov subspaces computed with either representation.

\begin{lemma}\label{lm:SylvesterEquivalence}
	Given the SE-DAE \eqref{eq:SE} and the underlying ODE \eqref{eq:underlying ODE}, the following equivalence between Sylvester equations holds
	\begin{subequations}
		\begin{align}
		&A
		\left[\begin{array}{c}
		V_1\\
		V_2
		\end{array}\right]
		- E
		\left[\begin{array}{c}
		V_1\\
		V_2
		\end{array}\right]
		S_V
		- B \,\Rtan = 0 \label{eq:Implicit-V}\\
		\iff &A_1 V_1 - E_1 V_1 S_V - B_1 \Rtan = 0 \label{eq:Explicit-V}
		\end{align}
	\end{subequations}
	where $V = \trans{\left[\trans{V_1},\trans{V_2}\right]}$ is partitioned according to $n_{dyn}$.
\end{lemma}
\begin{proof}
	\begin{equation*}
	\begin{aligned}
	&A
	\left[\begin{array}{c}
	V_1\\
	V_2
	\end{array}\right]
	- E
	\left[\begin{array}{c}
	V_1\\
	V_2
	\end{array}\right]
	S_V
	- B \Rtan = 0 \\
	\iff &\left\{\begin{array}{l}
	A_{11} V_1 + A_{12}V_2-E_{11}V_1 S_V - B_{11}\Rtan = 0\\
	V_2 = \inv{A_{22}}\left(-A_{21}V_1 + B_{22}\Rtan\right)
	\end{array}\right.\\
	\iff &A_1 V_1 - E_1 V_1 S_V - B_1 \Rtan = 0
	\end{aligned}
	\end{equation*}
\end{proof}
\noindent
Naturally, the dual result in terms of the output Sylvester equation \eqref{eq:W-Sylvester} holds as well. This result implies that the projection matrices $V_1$ and $W_1$ needed to approximate the underlying ODE can be computed  by solving large sparse system of equations in terms of the original DAE matrices without explicitly computing Schur complements. 
Even more importantly, this result implies, through the duality between Krylov and Sylvester, that both the reduced model obtained by reducing the underlying ODE using $V_1,W_1$ and the one obtained through the reduction of the DAE using $V,W$ share the same interpolation data $(S_V,\Rtan)$ and hence achieve tangential matching of the same moments of the original model.
However, this result does not guarantee that these reduced models will actually be the same.
In order to get this result, we have to ensure that the second procedure implicitly reduces the underlying ODE by operating on the DAE. 
We can analyze this by projecting the ODE and SE-DAE with their respective projection matrices and compare the expressions. The results are summarized in table \ref{tab:2sidedComparison}, where the last column includes the correction terms of theorem \ref{thm:Gugercin}.
\begin{table}[h!]
	\centering
	\begin{tabular}{r|c|c|c}
		& ODE $\left(V_1,W_1\right)$ & SE-DAE  $\left(V,W\right)$ & correction \cite{GugercinStykel_2013}\\ 
		\midrule
		$E_r$ & $\trans{W_1}E_{1}V_1$ & $\trans{W_1}E_{1}V_1$& -\\  
		$A_r$ & $\trans{W_1}A_1 V_1$ & 
		$\trans{W_1}A_1 V_1 \,\uwave{- \Ltan \Dr \Rtan} $ & $\Ltan \Dr \Rtan$\\
		$B_r$ & $\trans{W_1}B_1$ &  $\trans{V_1}B_1 \,\uwave{- \Ltan \Dr}$ & $\Ltan \Dr$\\
		$C_r$ & $C_1 V_1$ & $C_1 V_1 \,\uwave{- \Dr \Rtan}$ & $\Dr \Rtan$ \\
		$D_r$ &$D + \Dr$ & $D \,\uwave{\textcolor{white}{+ \Dr}}$ & $\Dr$\\
		\bottomrule
	\end{tabular} 
	\caption{Comparison of a skew projection for ODE and SE-DAE}
	\label{tab:2sidedComparison}
\end{table}

The results in the table show that the two reduced models differ exactly by the terms that are compensated in the SE-DAE aware reduction algorithm of theorem \ref{thm:Gugercin}, confirming that this procedure effectively yields moment matching and implicit reduction of the underlying ODE for a skew projection.

Will this result hold true also for one-sided reduction? 
The theory on ODE reduction suggests that this should be the case.
However, it turns out that this is not always true when operating with SE-DAEs, as we will state in the following theorem.

\begin{thm}\label{thm:orthogonalReduction4Stability}
	Consider a SE-DAE as in \eqref{eq:SE} and its underlying ODE defined in \eqref{eq:underlying ODE}.
	The ROM obtained by applying orthogonal projection ($W=V$) of the SE-DAE, being $V$ a basis of the Krylov subspace as in \eqref{eq:Implicit-V}, is equivalent to the one resulting from the respective orthogonal projection applied to the underlying ODE with $V_1$ satisfying \eqref{eq:Explicit-V}, provided that $B_{22}=0$.
\end{thm}
\begin{proof}
	The proof is obtained by straightforward computation of the orthogonal projection of the SE-DAE with $V = \trans{\left[\trans{V_1},\trans{V_2}\right]}$ and by using the relationship $V_{2} = \inv{A_{22}}\left(-A_{21}V_{1} + B_{22}\Rtan\right)$ obtained in lemma \ref{lm:SylvesterEquivalence}. The correction terms of table \ref{tab:2sidedComparison} will be applied. Since we only compute $V$ as an input Krylov subspace, the matrix of left tangential directions $\Ltan$ is set to 0. Subsequently, the ROM obtained is compared to the one that would result from a direct orthogonal projection of the underlying ODE. For brevity we will omit the computations and directly compare the reduced matrices in table \ref{tab:OrthogonalProjection}.
	\begin{table}[h!]
		\centering
		\begin{tabular}{r|c|c}
			& ODE ($W_1$=$V_1$) & SE-DAE ($W$=$V$), corrected \cite{GugercinStykel_2013}\\ 
			\midrule
			$E_r$ & $\trans{V_1}E_{11}V_1$ & $\trans{V_1}E_{11}V_1$ \\ 
			$A_r$ & $\trans{V_1}A_1 V_1$ & $\trans{V_1}A_1 V_1 
				\uwave{+\left[\trans{V_1}\left(A_{12}\inv{A_{22}}-\trans{A_{21}}\invt{A_{22}} \right)+ \trans{\Rtan}\trans{B_{22}}\invt{A_{22}} \right]B_{22}\Rtan}$ \\ 
			$B_r$ & $\trans{V_1}B_1$ & $\trans{V_1}
			\uwave{\left(B_{11}-\trans{A_{21}}\invt{A_{22}}B_{22}\right) + \trans{\Rtan}\trans{B_{22}}\invt{A_{22}}B_{22}}$\\ 
			$C_r$ & $C_1 V_1$ & $C_1 V_1$ \\ 
			$D_r$ & $D + \Dr$ & $D + \Dr$ \\ 
			\bottomrule
		\end{tabular} 
		\caption{Comparison of an orthogonal projection for ODE and SE-DAE}
		\label{tab:OrthogonalProjection}
	\end{table}
	
	The results show how this procedure generally fails to correctly reduce the underlying ODE. Nevertheless, if the original model satisfies $B_{22} = 0$, then the reduction is correct after all. 
\end{proof}

Accordingly, care needs to be taken when applying orthogonal projection to reduce DAEs. Even though moment matching is still guaranteed by the computation of the Krylov subspace, the reduction fails to project the underlying ODE unless the condition $B_{22}=0$ is satisfied. Even more importantly: the additional terms in table \ref{tab:OrthogonalProjection} might cause loss of dissipativity of the reduced model!
The relevance of this will be shown through a numerical example in section \ref{sec:examples}.
Further note that the assumption $B_{22}=0$ is relevant from a practical standpoint but restrictive from a theoretical one, since the results in table \ref{tab:OrthogonalProjection} infer that a system satisfying $A_{22}\ts=\ts\trans{A_{22}}, A_{12}\ts=\ts\trans{A_{21}}, C_{22}\ts=\ts\trans{B_{22}}$ could be correctly reduced through orthogonal projection by choosing $\Ltan\ts=\ts \trans{\Rtan}$ in the correction terms.	In fact, in this case $\trans{\Rtan}\trans{B_{22}}\invt{A_{22}}B_{22}\ts=\ts -\Ltan\Dr$ and $B_{11}-\trans{A_{21}}\invt{A_{22}}B_{22}\ts=\ts B_1$.

Finally, note that this result extends naturally to the dual case of output-based orthogonal projection.
\begin{corollary}\label{thm:OutputOrthogonal}
	Consider an SE-DAE as in \eqref{eq:SE} and its underlying ODE defined in \eqref{eq:underlying ODE}.
	The ROM obtained by applying orthogonal projection $(V\ts=\ts W)$ of the SE-DAE, being $W$ a basis for the Krylov subspace as in \eqref{eq:W-Sylvester}, is equivalent to the one resulting from the respective orthogonal projection applied to the underlying ODE, provided that $C_{22}\ts=\ts0$ or $A_{22}\ts=\ts\trans{A_{22}}, A_{12}\ts=\ts\trans{A_{21}}, C_{22}\ts=\ts\trans{B_{22}}$.
\end{corollary}
\begin{proof}
	The proof is analogous to the dual case. Note that also in this case, the condition $C_{22}\ts=\ts0$ can be dropped if the system satisfies $A_{22}\ts=\ts\trans{A_{22}}, A_{12}\ts=\ts\trans{A_{21}}, C_{22}\ts=\ts\trans{B_{22}}$ and the right tangential directions for the correction terms are chosen such that $\Rtan\ts=\ts\trans{\Ltan}$.
\end{proof}

To sum up, provided that the SE-DAE system considered is not excited through the algebraic equations or the output influenced by algebraic variables, then DAE-aware reduction by orthogonal projection as proposed in theorem \ref{thm:Gugercin} can effectively reduce the underlying ODE implicitly, by choosing the right type of Krylov subspace (input or output). If the system has some symmetries with respect to the algebraic variables, namely $A_{22}\ts=\ts\trans{A_{22}}, A_{12}\ts=\ts\trans{A_{21}}, C_{22}\ts=\ts\trans{B_{22}}$, then both input and output Krylov subspaces can be used, provided the tangential directions are chosen such that $\Rtan\ts=\ts \trans{\Ltan}$.

\subsection{Stability-preserving reduction by orthogonal projection}
Due to the previous results, we understand that we can implicitly apply orthogonal reduction to the underlying ODE choosing either input or output Krylov subspaces depending on the structure of the system. As we know from ODE theory, this process preserves stability in case that the underlying ODE---and hence by definition \ref{def:SD_DAE} also the SE-DAE---is given in strictly dissipative form.
\begin{thm}
	Assume the SE-DAE as in \eqref{eq:SE} is given in strictly dissipative form. Further assume that either $B_{22}\ts=\ts0$ or $A_{22}\ts=\ts\trans{A_{22}}, A_{12}\ts=\ts\trans{A_{21}}, C_{22}\ts\ts=\trans{B_{22}}$. Then reduction of the SE-DAE through orthogonal projection using the input Krylov subspace defined in \eqref{eq:Implicit-V} and the DAE-aware procedure of theorem \ref{thm:Gugercin} yields an asymptotically stable ROM that implicitly reduces the underlying ODE.
\end{thm}
\begin{proof}
	The proof is obtained by observing that due to lemma \ref{lm:SylvesterEquivalence}, the proposed reduction of the SE-DAE is equivalent to the direct reduction of the underlying ODE. Applying orthogonal projection preserves strict dissipativity and hereby stability.
\end{proof}
Naturally the dual result holds as well.
\begin{corollary}
	Assume the SE-DAE as in \eqref{eq:SE} is given in strictly dissipative form. Further assume that either $C_{22}\ts=\ts0$ or $A_{22}\ts=\ts\trans{A_{22}}, A_{12}\ts=\ts\trans{A_{21}}, C_{22}\ts=\ts\trans{B_{22}}$.  Then reduction of the SE-DAE through orthogonal projection using an output Krylov subspace and the DAE-aware procedure of theorem \ref{thm:Gugercin} yields an asymptotically stable ROM that implicitly reduces the underlying ODE.
\end{corollary}		
The effectiveness of this method will be assessed through numerical examples in section \ref{sec:examples}. 
In addition, it will be shown that if the correct choice of Krylov subspace is ignored, the ROM might indeed become unstable. 

\section{Stability preserving, adaptive reduction by $\Htwo$-pseudo-optimality}\label{sec:Stability_Preserving_H2_pseudo}
The preservation of stability in the case of ODEs can be achieved by construction using $\Htwo$-pseudo-optimal reduction as presented in section \ref{sec:H2-opt}. 
In this procedure, if the shifts of the Krylov subspaces are chosen on the right complex half-plane, then the ROM will have all eigenvalues in the left complex half-plane and hence be asymptotically stable.
In the following, we will extend this result to the case of SE-DAEs. Since the choice of shifts becomes twice as important in the pseudo-optimal setting, we will subsequently address the question of how to appropriately select the shifts.

\subsection{$\Htwo$-pseudo-optimal reduction of SE-DAEs} \label{sec:SE-DAE PORK} 
Recall the construction of pseudo-optimal ROMs following the PORK algorithm (cf. algorithm \ref{alg:pork}). 
One peculiar characteristic of PORK is that the reduced matrices $(E_r,A_r,B_r)$---or $(E_r,A_r,C_r)$ in the dual version---are \emph{independent} of the original model and depend only on the interpolation data, i.e. merely on the reduced eigenvalues and tangential directions encoded in the pair ($S_V,\Rtan $) or ($S_W,\Ltan $) respectively. This characteristic makes it particularly suitable for the extension to SE-DAEs, at least in case the reduced model should be an ODE and not a DAE. Note however, that since the DAEs considered are of index $\nu = 1$, their transfer behavior can be modeled entirely by an equivalent ODE, as it can be seen from the transfer function in \eqref{eq:DAE_tf}. Therefore, the restriction to reduced ODE models can be conducted without loss of generality.

It turns out that the PORK algorithm can be adapted to SE-DAEs by combining it with the appropriate correction terms proposed in theorem \ref{thm:Gugercin}. Since the proof is straightforward following what has been said so far, we will limit the exposition to the adapted algorithm.

\addtocounter{algorithm}{+1}
\begin{algorithm}[!ht]\caption{SE-DAE PORK}\label{alg:SE-DAE PORK}
	\begin{algorithmic}[1]
		\Require ($E$, $A$, $B$, $C$, $D$), $(S_V, \Rtan)$
		\Ensure $\mathcal{H}_2$-pseudo-optimal reduced system matrices
		\State {$V \gets \ A V - E V S_V - B \Rtan = 0$}
		\hfill{// Krylov subspace}
		\State {$\inv{P_r} =$ lyap$(-\trans{S_V}, \trans{\Rtan}\Rtan)$}  \hfill{// As in the ODE case}
		\State {$B_r = -P_r \trans{\Rtan}$}
		\State {$A_r = S_V + B_r \Rtan$, $\;E_r = I$} 
		\State {$C_r = CV + \Dr \Rtan$, $\;D_r = D + \Dr$} \hfill{// cf. theorem \ref{thm:Gugercin} and \eqref{eq:P(s)}}
	\end{algorithmic}
\end{algorithm}
Naturally this result extends to the dual case.
This stability preserving reduction procedure clearly is far less restrictive than the one proposed in section \ref{sec:StabPrev1-orthogonal}. On the one hand, it does not require the original model to be given in strictly dissipative form. On the other hand, it poses no restrictions on the structure of the original model and can therefore be applied to any SE-DAEs.

What still remains unclear at this point is how to choose appropriate shifts. Recall that in standard reduction based on Krylov subspaces the choice of shifts only affects the matching frequencies, while the eigenvalues implicitly result from the reduction. By contrast, in the pseudo-optimal setting the reduced eigenvalues are directly chosen through the choice of shifts. While it is known that pseudo-optimality is a necessary condition for $\Htwo$-optimality (cf. Maier-Luenberger conditions \cite{Beattie_2014_Survey}), this does not imply that pseudo-optimal ROMs are better than the equivalent ROM achieved, for instance, through two-sided reduction. 
The choice of shifts becomes even more crucial and the next section will show how these can be determined adaptively even for SE-DAEs within the reduction framework CUREd SPARK.

\subsection{Judicious choice of shifts: CUREd SPARK for SE-DAEs} \label{sec:SE-DAE CUREd SPARK}
Analogously to the ODE case, the cumulative reduction framework with adaptive choice of shifts (CUREd SPARK) can be extended to the case of SE-DAEs to complement the pseudo-optimal reduction with a judicious choice of shifts. With this respect, it is worth separating the discussion for the two complementary classes of SE-DAEs: the ones presenting and the ones omitting the implicit feedthrough term $\Dr$.

Whenever the SE-DAE does not have an implicit feedthrough term $\Dr$ as defined in \eqref{eq:P(s)}, then the original CUREd SPARK introduced in section \ref{sec:H2-opt} can be used without restrictions. In fact, the algorithms do not require the matrix $E$ to be regular. Numerical examples in section \ref{sec:examples} will show the effectiveness of this procedure. It is worth noting at this point that the vast majority of benchmark SE-DAEs models commonly available fall into this category. This is motivated by physical intuition, since real technical systems mostly have some delay between excitation and response and are hence feedthrough-free.

On the other hand, if an implicit feedthrough term $\Dr$ is present, then an adaptation of CUREd SPARK is required. It can be shown that the iteration in CURE \mbox{(cf. \eqref{eq:CURE})} can be adapted to take care of the implicit feedthrough term $\Dr$ resulting at each iteration $k$. However, the greater issue arises when applying SPARK to SE-DAEs and requires a modification of the DAE, as we shall discuss in the following.

As seen in section \ref{sec:H2-opt}, the goal of SPARK is to find a reduced order model $G_r(s)$ of order n=2 that is locally a $\Htwo$-optimal approximation of the original model $G(s)$. Exploiting pseudo-optimality, the error norm is given by (cf. \eqref{eq:H2errorPO})
\begin{equation*}
\norm{G_e(s)}_{\Htwo}^2 \stackrel{\Htwo-po}{=} \norm{G(s)}_{\Htwo}^2 - \norm{G_r(s)}_{\Htwo}^2
\end{equation*}
This was true at least in the ODE case, where it is common practice to leave feedthrough terms $D$ aside from the reduction and integrating them subsequently into the reduced model. However, in the case of DAEs retaining an implicit feedthrough term $\Dr$, this term is hidden inside the matrices $A,B,C$ and cannot be directly removed prior to the reduction. In this case, obviously the above error expression becomes meaningless since the $\Htwo$-norm of a non strictly proper system does not exist.
To better understand this fact, let us separate the strictly proper part $G^{sp}(s)$ of the transfer function $G(s)$ from the implicit feedthrough $\Dr$ and assume that the reduced model $G_r(s)$ is also given as the sum of a strictly proper part $G_r^{sp}(s)$, which is a pseudo-optimal approximant of $G^{sp}(s)$, and the same feedthrough term $\Dr$. The error expression then becomes
\begin{align}
\norm{G_e(s)}_{\Htwo}^2 &\mathrel{\makebox[\widthof{$\overset{\Htwo-po}{=}$}]{=}} \innProd{G(s)-G_r(s)}{G(s)-G_r(s)}_{\Htwo} \nonumber\\
&\mathrel{\makebox[\widthof{$\overset{\Htwo-po}{=}$}]{=}} \innProd{G^{sp}(s)+\Dr - G_r^{sp}(s) - \Dr}{G(s)-G_r(s)}_{\Htwo} \nonumber\\
&\stackrel{\Htwo-po}{=} \norm{G^{sp}(s)}_{\Htwo}^2 - \norm{G^{sp}_r(s)}_{\Htwo}^2 \label{eq:H2-error-DAE}
\end{align}
In this case, the feedthrough term cancels out and by pseudo-optimality we obtain the difference of two well defined norms. Therefore, in order for \eqref{eq:H2-error-DAE} to hold true, we need to: 
\begin{enumerate}[a)]
	\item find an expression for $G^{sp}(s)$, the strictly proper part of $G(s)$,
	\item find a pseudo-optimal approximation $G_r^{sp}(s)$ for $G^{sp}(s)$,
	\item make sure the reduced model $G_r(s)$ retains the feedthrough term $\Dr$, i.e. $G_r(s) = G_r^{sp}(s) + \Dr$.
\end{enumerate}

Whenever the SE-DAE has an implicit feedthrough term $\Dr$, the model reduction procedure has to compute it and take it into consideration during reduction (cf. \mbox{theorem \ref{thm:Gugercin}} or \cite{GugercinStykel_2013}). 
Therefore, its inclusion in the reduced model is easy to implement. Further, provided a strictly proper model $G^{sp}(s)$ is given, all the algorithms discussed so far can be implemented to reduce it. It follows that the only open question at this point remains how to extract a strictly proper representation $G^{sp}(s)$ out of an SE-DAE with implicit feedthrough. One possible solution is presented in the following:

\begin{proposition}
	Consider a SE-DAE as in \eqref{eq:SE} with an implicit feedthrough term $\Dr\ts\neq\ts0$, whose transfer function can be written as the $G(s)\ts=\ts G^{sp}(s)\ts+\ts\Dr$. Then a state-space realization of the strictly proper part $G^{sp}(s)$ is given by
	\begin{equation}
	\begin{aligned}
	\left[
	\begin{array}{cc}
	E_{11} & 0\\
	0 & 0
	\end{array}
	\right]
	\left[\begin{array}{c}
	\dot{x}_1\\
	\dot{x}_2
	\end{array}\right]
	&= 
	\left[
	\begin{array}{cc}
	A_{11} & A_{12}\\
	A_{21} & A_{22}
	\end{array}
	\right]
	\left[\begin{array}{c}
	x_1\\
	x_2
	\end{array}\right]
	+ 
	\left[\begin{array}{c}
	B_{11}- A_{12}\inv{A_{22}}B_{22}\\
	0
	\end{array}\right]
	u \\
	y &= \left[C_{11}, C_{22}\right] 
	\left[\begin{array}{c}
	x_1\\
	x_2
	\end{array}\right]
	\end{aligned}
	\label{eq:SE_sp}
	\end{equation}
	or alternatively 
	
	\begin{equation}
	\begin{aligned}
	\left[
	\begin{array}{cc}
	E_{11} & 0\\
	0 & 0
	\end{array}
	\right]
	\left[\begin{array}{c}
	\dot{x}_1\\
	\dot{x}_2
	\end{array}\right]
	&= 
	\left[
	\begin{array}{cc}
	A_{11} & A_{12}\\
	A_{21} & A_{22}
	\end{array}
	\right]
	\left[\begin{array}{c}
	x_1\\
	x_2
	\end{array}\right]
	+ 
	\left[\begin{array}{c}
	B_{11}\\
	B_{22}
	\end{array}\right]
	u \\
	y &= \left[C_{11} - C_{22}\inv{A_{22}}A_{21}, 0\right] 
	\left[\begin{array}{c}
	x_1\\
	x_2
	\end{array}\right]
	\end{aligned}
	\label{eq:SE_sp2}
	\end{equation}
\end{proposition}
\begin{proof}
	The proof is straightforward and follows by computing the underlying ODE for both the system in \eqref{eq:SE_sp} and \eqref{eq:SE_sp2} as it was done in \eqref{eq:underlying ODE}. 
\end{proof}
Note that, if an implicit feedthrough term $\Dr\ts=\ts -C_{22}\inv{A_{22}}B_{22}$ is present in the system, then the reduced model must pertain it in order to have a bounded approximation error. This implies the fact that the term $\Dr$ must be computed anyways. Therefore, the vectors $\inv{A_{22}}B_{22}$ or $C_{22}\inv{A_{22}}$ required for computing the strictly proper realization $G^{sp}(s)$ are already available without further computations.

Exploiting this result, it is possible to choose reduced order and interpolation points adaptively by means of CUREd SPARK even for SE-DAEs with implicit feedthrough term $\Dr$. We will illustrate the effectiveness of this approach through numerical examples in the following section.

\section{Validation through numerical examples}\label{sec:examples}
The theoretical results presented in the previous sections are validated through some numerical examples. First, we introduce the models used for the computations, then we discuss the results.
\subsection{Transmission line model}
A transmission line, i.e. the transmission of electrical signals through a one-dimensional conductor, is described by a partial differential equation (PDE) in both time and space. A common way of approximating this PDE is by using basic elements of an electric network distributed along the line, discretizing the PDE in space \cite{Reeve_1995}. A typical, simplified\footnote{The simplification results form the omission of the conductance $G_i'$ that is usually shunt in parallel to the conductances in such a network and is justified by the selection of parameters (cf. \cite[p.588]{Reeve_1995})}  representation of such a discretized model is given in figure \ref{fig:TransmissionLine}.

\begin{figure}[h]
	\centering
	\includegraphics[height=3.5cm]{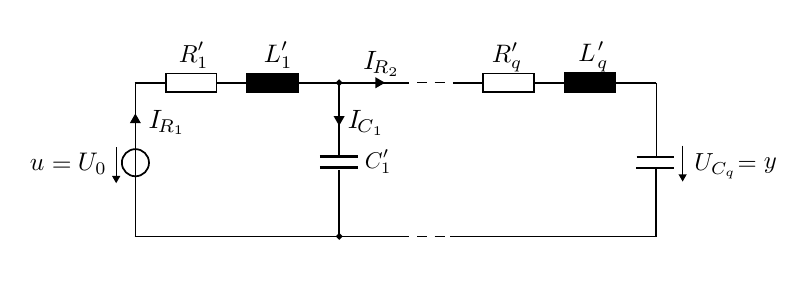}
	\caption{Electrical circuit of a transmission line approximation}
	\label{fig:TransmissionLine}
\end{figure}
The input of the system is represented by a voltage source $U_0$. $R'_i, L'_i, C'_i$ represent the distributed resistance, inductance and capacitance of the line respectively. The typical output of the system is the voltage $U_{C_q}$ over the capacitor  at the end of the line.
The number of loops used to discretize the transmission line shall be denoted by $q$. In general, the higher the number of loops, the better the approximation.
The values for the distributed parameters are taken from \cite[p.588]{Reeve_1995}. In general, the parameters vary largely depending on the geometry and materials of the problems. For our simulations, the special case of a telephone line and a transmission frequency of 1 Hz was taken, for which the parameters are given in table \ref{tab:TransmissionLine}.
\begin{table}[h]
	\centering
	\begin{tabular}{ll}
		\toprule
		Parameter	& Value \\ 
		\midrule
		$R'_i$	& \unitfrac[$172.24 \cdot 10^{-3}$]{$\Omega$}{m}  \\ 
		$L'_i$	& \unitfrac[$0.61 \cdot 10^{-6}$]{H}{m} \\ 
		$C'_i$	& \unitfrac[$51.57 \cdot 10^{-12}$]{F}{m}\\ 
		\bottomrule
	\end{tabular} 
	\caption{Transmission line parameters used for simulations}
	\label{tab:TransmissionLine}
\end{table}

The modeling is conducted using first principles, such as Kirchhoff's laws and constitutive equations for the elements (cf. e.g. \cite{Hambley_2005}). Each loop is modeled using the variables $x_i = \trans{\left[I_{R_i},U_{C_i},U_{R_i},I_{C_i}, U_{L_i}\right]}$, which can be stacked up to yield the $5q$-dimensional state vector that characterizes the DAE. The general structure of the SE-DAE matrices resulting from this modelling can be found in \ref{appx:TL}. 
Note that for the standard choice of input and output as in figure \ref{fig:TransmissionLine} the system matrices satisfy $B_{22}\ts\neq\ts0$ (the input $U_0$ enters the equations through Kirchhoff's law, an algebraic constraint), $C_{22}\ts=\ts0$ (the output is a dynamic variable) and $A_{22}\ts\neq\ts\trans{A_{22}}$.
Also note that even though the modeling does not yield directly a strictly dissipative representation, the procedure described in theorem \ref{thm:LMI} was used to find a suitable representation up until an order of $N$=700, which corresponds to $q$=140 loops. 
Finally, note that by appropriately selecting the output of the system, e.g. by choosing the voltage over the first inductance $L'_1$, it is also possible to model a system with implicit feedthrough $\Dr$. 

\subsection{Power system benchmark models}\label{sec:powersystems}
Another set of benchmark systems chosen to test the validity of the proposed algorithms is given by the power system examples created at the Brazilian Electrical Energy Research Center (CEPEL) and available online in the MOR Wiki \cite{MORWiki_power}.
These systems represent large power systems (including lines, buses, power plants etc.) linearized about an operating point and are used to simulate and study the oscillations of complex power systems. The reduction of such systems is relevant for the numerical simulations required, for instance, for small-signal stability analysis, controller design, and real-time investigations of the transient behavior. For more informations on the origin of the systems, please refer to \cite{Rommes_2006, MORWiki_power}.

Most of these systems are index 1 DAEs that are either already in semi-explicit form or can be transformed to it by simple reordering of rows and columns.		
The models are mainly strictly proper, meaning that they have neither explicit feedthrough $D$ nor implicit feedthrough $\Dr$. However, the so called ``MIMO46'' system has SISO transfer functions on the diagonal of the transfer function matrix that present such an implicit feedthrough $\Dr\ts\neq\ts0$.

For the numerical investigations of this treatise, two different models of the Brazilian Interconnected Power System of 1997 (BIPS/97) were used. Both of them have a state-space dimension of roughly $N$=13250 and a dynamic order, that is the order of the underlying ODE, of $n_{dyn}$=1664. While the first one is SISO and strictly proper ($ww\_vref\_6405.mat$), the second one is MIMO with implicit feedthrough on the diagonal channels ($mimo46x46\_system.mat$).

\subsection{Results}
\subsubsection{The importance of SE-DAE-awareness in orthogonal reduction} 
The first example is aimed at showing what can happen if the results of section \ref{sec:OrthogonalProjectionOfDAEs} are neglected. Figure \ref{fig:DAE-aware-MOR} shows the results obtained on a transmission line model with $q\ts=\ts10$ loops. 

\begin{figure}[h!]
	\centering
	\includegraphics[width = \mywidth]{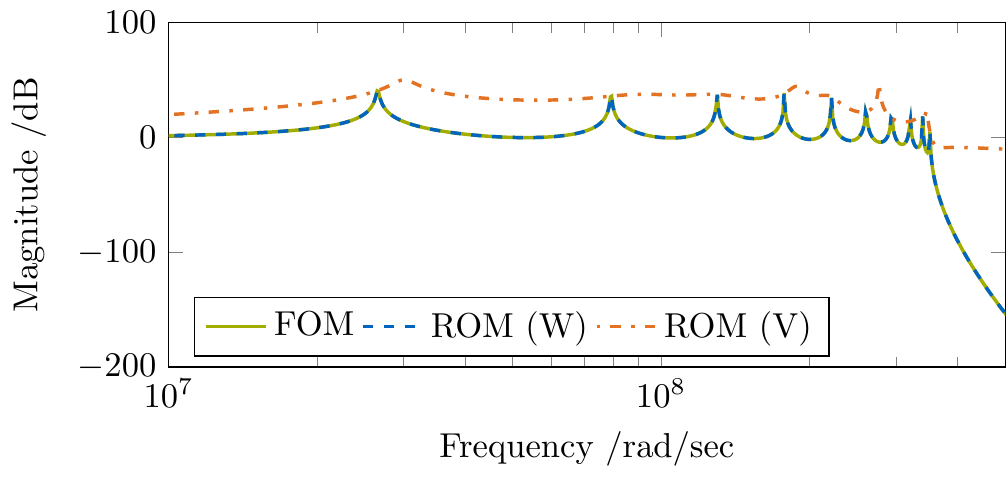}
	\caption{Orthogonal projection using input and output Krylov subspaces for the transmission line example ($q=10$)}
	\label{fig:DAE-aware-MOR}
\end{figure}
The bode plot of the FOM with $N$=50 is compared to two ROMs of order $n$=20, one obtained through V-based and the other through W-based orthogonal projection of the SE-DAE. Recall that since $C_{22}\ts=\ts0$, $B_{22}\ts\neq\ts 0$ and $A_{22}\neq\trans{A_{22}}$, only W-based orthogonal projection is guaranteed to effectively reduce the underlying ODE. The shifts for the Krylov subspaces were chosen along the imaginary axis with frequencies corresponding to the peaks of the FOM. As it can be seen, the W-based ROM matches the FOM perfectly. In fact, 20 corresponds to the dynamic order of the system, i.e. the order of the underlying ODE. Therefore, the W-based ROM is merely a transformation of the underlying ODE and hence shares the same transfer function. On the other hand, the figure shows how this is not true for the V-based ROM. Moment matching at the shifts still holds, however, the reduction results are clearly worse. In particular, this is not a transformation of the underlying ODE anymore.

\subsubsection{Stability preserving reduction of strictly dissipative SE-DAEs}
The effectiveness in preserving stability while interpolating the underlying ODE of the reduction method proposed in theorem \ref{thm:orthogonalReduction4Stability} is shown on a transmission line model with $q$=140 loops and state-space dimension $N$=700. Before reduction, the SE-DAE has been transformed to a strictly dissipative formulation as explained in theorem \ref{thm:LMI}. The results are given in figure \ref{fig:TL-MOR}. The FOM is compared with two ROMs of order $n$=100, one obtained with V-based, the other with W-based orthogonal projection about the origin. Recall that this system satisfies $C_{22}=0$, hence only $W$-based orthogonal projection is expected to yield acceptable results.
\begin{figure}[h!]
	\centering
	\includegraphics[width = \mywidth]{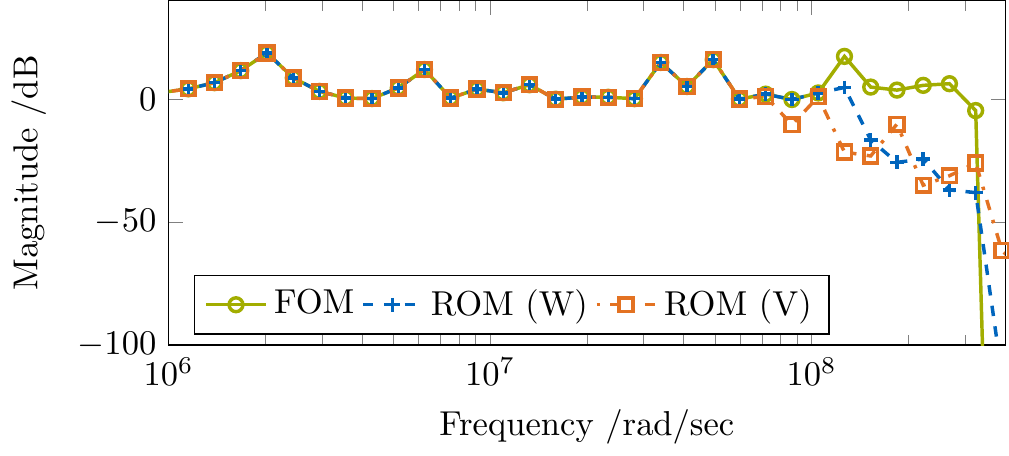}
	\caption{Orthogonal reduction using input and output Krylov subspaces for the transmission line example ($q=140$)}
	\label{fig:TL-MOR}
\end{figure}

The plot suggests that the W-based ROM yields a slightly better approximation of the FOM at higher frequency range. The plot does not show that while the W-based ROM coincides with an equivalent ROM obtained reducing directly the underlying ODE, the V-based one does not, emphasizing again the importance of SE-DAE-aware reduction. Even more importantly: while the W-based ROM preserves strict dissipativity and stability, the V-based ROM is unstable! These results are summarized in table \ref{tab:TL-MOR}.
\begin{table}[h]
	\centering
	\begin{tabular}{c|ccc}
		& FOM	& ROM (W) & ROM (V) \\ 
		\midrule
		dissipative & \Checkmark 	&\Checkmark & \ding{53}  \\ 
		stable	& \Checkmark & \Checkmark & \ding{53}\\ \addlinespace[2mm]
	\end{tabular} 
	\caption{Comparison of different orthogonal projections of the transmission line}
	\label{tab:TL-MOR}
\end{table}

\subsubsection{Stability preserving, $\Htwo$-pseudo-optimal reduction of general SE-DAE}
The stability preserving, adaptive reduction of the SE-DAE with CUREd SPARK is possible even for systems that are not in strictly dissipative form. As discussed in section \ref{sec:Stability_Preserving_H2_pseudo}, whenever the system has no implicit feedthrough $\Dr$, CUREd SPARK can be applied by simply replacing the PORK algorithm for $\Htwo$-pseudo-optimal reduction by the SE-DAE PORK (algorithm \ref{alg:SE-DAE PORK}).		 
The reduction results using this strategy for the strictly proper power system described in \ref{sec:powersystems} are shown in figure \ref{fig:BIPS/97}, where the magnitude plot of the FOM (in dB) is compared to the magnitudes of the error systems resulting from a $W$ or $V$-based reduction through CUREd SPARK respectively.

\begin{figure}[h!]
	\centering
	\includegraphics[width = \mywidth]{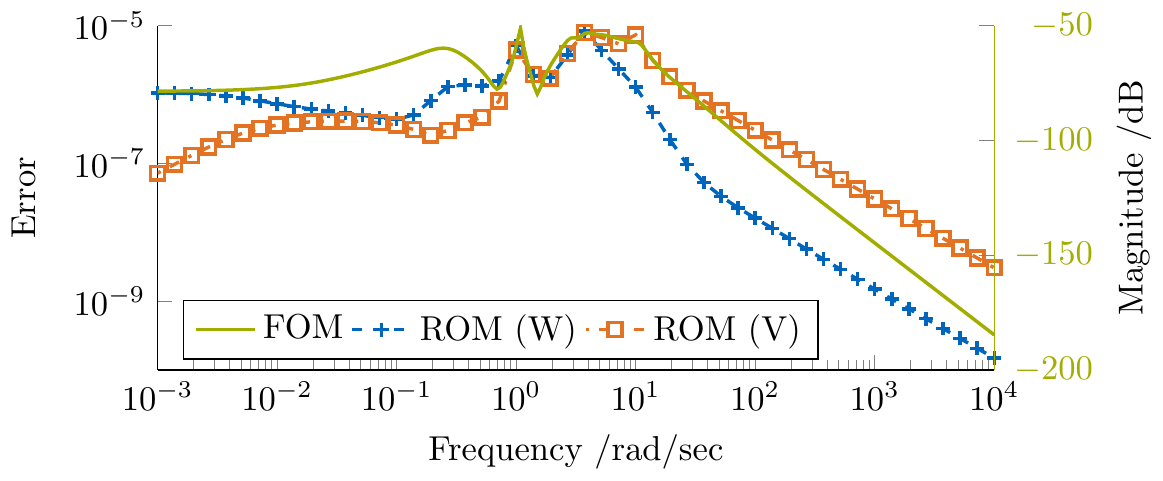}
	\caption{CUREd SPARK reduction of the BIPS/97 system ($\Dr=0$, $n=50$)}
	\label{fig:BIPS/97}
\end{figure}
The original order of the DAE is $N$=13251, the dynamic order is $n_{dyn}=1664$. The reduction was conducted using CUREd SPARK to a reduced order of n=50, both with $V$ and $W$-based factorization. The selection of the shifts was conducted automatically in each iteration of CURE by the greedy, $\Htwo$-optimal algorithm in SPARK. As it can be seen from the magnitude of the frequency response, both reduced models achieve a satisfactory approximation of the FOM while preserving stability.

The results of the reduction of SE-DAEs presenting an implicit feedthrough are shown in figure \ref{fig:BIPS/97_MIMO1} and figure \ref{fig:BIPS/97_MIMO2}. The FOM represents one of the diagonal input-output channels of the MIMO power system described in section \ref{sec:powersystems}. Like the previous system, the original order is $N$=13250 and the dynamic order is $n_{dyn}$=1664. The reduced order was chosen to $n$=50 for all channels considered. Also in this case, the choice of shifts was conducted automatically at each iteration of CURE by the greedy algorithm in SPARK.

\begin{figure}[h!]
	\centering
	\includegraphics[width = \mywidth]{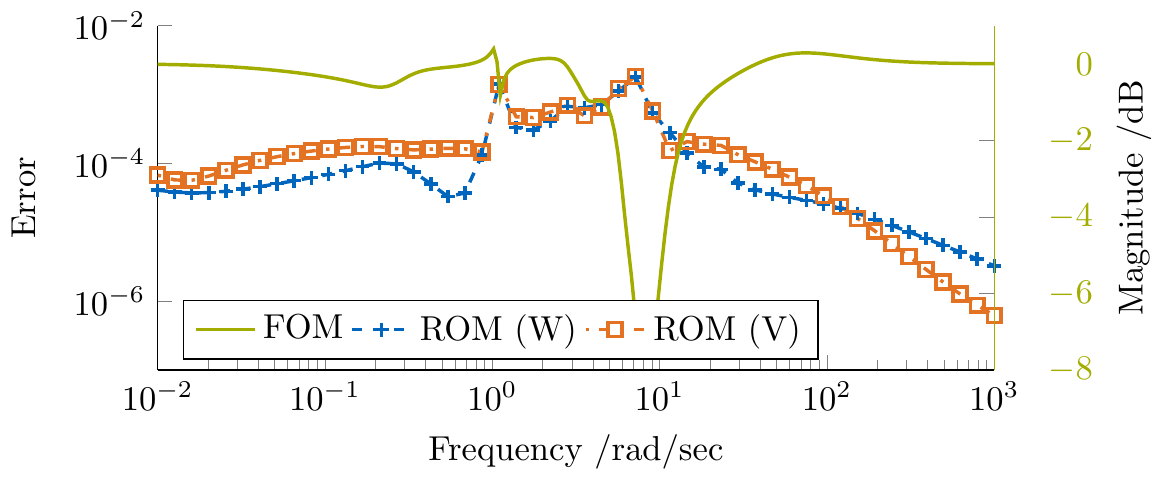}	
	\caption{CUREd SPARK reduction of the MIMO BIPS/97, channel 25 ($\Dr\neq 0, n=50$)}
	\label{fig:BIPS/97_MIMO1}
\end{figure}
\begin{figure}[h!]
	\centering
	\includegraphics[width = \mywidth]{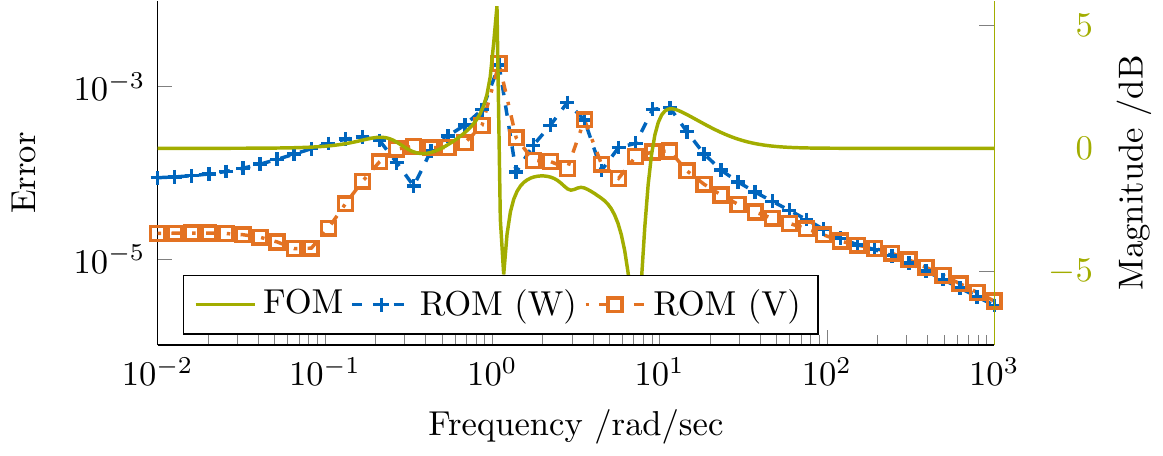}	
	\caption{CUREd SPARK reduction of the MIMO BIPS/97, channel 42 ($\Dr\neq 0, n=50$)}
	\label{fig:BIPS/97_MIMO2}
\end{figure}
The figure shows how the proposed reduction algorithm manages to detect the relevant dynamics of the systems and preserve both stability and the implicit feedthrough.	

\section{Conclusions}
In this contribution the stability preserving reduction of differential-algebraic equations (DAE) of index 1 in semi-explicit form has been discussed. To the authors' knowledge, this is the first work to address this issue. To this purpose, two different methods have been presented. 

The first is aimed at preserving the strictly dissipative form of the underlying ordinary differential equation (ODE) by means of orthogonal projection of the DAE. A transformation of the DAE to strictly dissipative form through the solution of a convex feasibility problem has been introduced. It has been shown in theory and through numerical examples that the correct orthogonal projection can be achieved only if the DAE fulfills certain conditions. In this case, the appropriate choice of input or output Krylov subspace for the orthogonal projection yields a strictly dissipative, hence stable reduced order model.

The second method uses $\Htwo$-pseudo-optimal reduction and is more general in that it does not impose any conditions on the DAE. For this purpose, the Pseudo-Optimal Rational Krylov (PORK) algorithm was adapted to include this class of DAEs. By construction, $\Htwo$-pseudo-optimal reduction generates asymptotically stable reduced order models if the shifts of the Krylov subspaces are chosen on the right complex half-plane. However, in this setting the choice of appropriate shifts is even more important for a good approximation quality. If the DAE has no implicit feedthrough term, which is often the case for technical systems, then this adapted PORK algorithm can be used within CUREd SPARK to adaptively select reduced order and shifts. If the system at hand does possess an implicit feedthrough term, then a realization of the strictly proper part was presented. This strictly proper part can be reduced adaptively and stability preserving with CUREd SPARK and the complete feedthrough term is added at the end.

The proposed methods have been tested numerically showing their validity on different academic and real-life systems.


\appendix
\section{Proof of proposition \ref{PROP:REINSCH}}\label{appx:Reinsch}
\begin{proof}
	The proof of proposition \ref{PROP:REINSCH} amounts to showing that the Schur complement of a strictly dissipative matrix, i.e. satisfying $A\ts+\ts\trans{A}\ts\nd\ts0$, is strictly dissipative itself. For this purpose, we shall partition the matrix $A$ according to \eqref{eq:SE}
	\[
	A = \left[
	\begin{array}{cc}
	A_{11} & A_{12}\\
	A_{21} & A_{22}
	\end{array}
	\right]
	\]
	and define the Schur complement  $A_1 \defeq A_{11} - A_{12}\inv{A_{22}}A_{21}$.
	It then follows 
	\[ A + \trans{A} \nd 0 \iff \invt{A}\left(A + \trans{A}\right)\inv{A} = \inv{A} + \invt{A} \nd 0 \]
	where the invertibility of A follows from its strict dissipativity.
	The inverse of A can be build block-wise as follows ($*$ denotes entries that are not required for our reasoning)
	\begin{equation}
	\inv{A} = \left[
	\begin{array}{cc}
	\inv{A_1} & * \\
	* & *
	\end{array}
	\right]
	\end{equation}
	and the existence of $\inv{A_1}$ is guaranteed by the existence of $\inv{A}$. It follows
	\begin{equation}
	\begin{aligned}
	& \inv{A} + \invt{A} =  \left[
	\begin{array}{cc}
	\inv{A_1} + \invt{A_1} & * \\
	* & *
	\end{array}
	\right] \nd 0\\
	\Longrightarrow&\inv{A_1} + \invt{A_1} \nd 0\\
	\iff& A_1\left(\inv{A_1} + \invt{A_1}\right)\trans{A_1} = A_1 + \trans{A_1} \nd 0 \\
	\end{aligned}
	\end{equation}
\end{proof}

\section{Transmission line SE-DAE}\label{appx:TL}
In the following we introduce the model of the transmission line presented in figure \ref{fig:TransmissionLine}. Since the aim is to analyze the characteristics of the DAE-aware procedures presented in this work, the model has been kept as simple as possible. The conductance $G_i$ that is typically in parallel to the capacitance $C_i$ has been neglected. This case corresponds e.g. to the transmission line model of a 24 gauge telephone PIC cable at 21$^\circ C$ and a 1 Hz frequency \cite{Reeve_1995}. The equations have been derived using first principles of physics due to their simplicity and direct physical interpretation \cite{Hambley_2005}. Each loop $i$ out of $q$ has a set of 5 states, namely $x_i = \trans{\left[I_{R_i},U_{C_i},U_{R_i},I_{C_i}, U_{L_i}\right]}$, where we have already made use of the fact that the current flowing through the resistances and inductances in each loop must be identical.
The five equations that model each loop and define the relations amongst the states are Kirchhoff's laws and the constitutive equations of each element (resistor, inductor, capacitor). For each loop $i=2,\dots,q$  it therefore holds
\begin{subequations}
	\begin{align}
	U_{R_i} + U_{L_i} + U_{C_i} - U_{C_{i-1}} &= 0 \label{eq:Kirchhoff_loop}\\ 
	I_{R_{i}} - I_{R_{i+1}} - I_{C_i} &=0\\
	U_{R_i} - R'_i I_{R_i} &=0\\
	U_{L_i} - L'_i\diff{I_{R_i}}{t} &=0\\
	I_{C_i} - C'_i\diff{U_{C_i}}{t}&=0
	\end{align}
\end{subequations}
For $i=1$, equation \eqref{eq:Kirchhoff_loop} is changed replacing $U_{C_{i-1}}$ by the input $U_0$.
This formalism leads to the SE-DAE defined block-wise as
\begin{equation}
\begin{aligned}
&\begin{pmat}[{.|}]
L_q & 0_q & 0_{q\times 3q} \cr
0_q & C_q & 0_{q\times 3q} \cr \-
0_{3q\times q} & 0_{3q\times q} & 0_{3q} \cr
\end{pmat}
\begin{pmat}[{}]
\dot{I}_{R} \cr
\dot{U}_C \cr \-
\dot{U}_R \cr 
\dot{I}_C \cr 
\dot{U}_L \cr
\end{pmat}=\\ 
&=\begin{pmat}[{.|..}]
0_q & 0_q & 0_q & 0_q & I_q \cr 
0_q & 0_q & 0_q & I_q & 0_q \cr \-
-R_q & 0_q & I_q & 0_q & 0_q \cr 
I_q - I^+_q & 0_q & 0_q & -I_q& 0_q \cr 
0_q & I_q - I^-_q & I_q & 0_q & I_q \cr
\end{pmat}
\begin{pmat}[{}]
I_{R} \cr
U_C \cr \-
U_R \cr 
I_C \cr 
U_L \cr
\end{pmat}
- 
\begin{pmat}[{}]
0_{q\times 1} \cr
0_{q\times 1} \cr \-
0_{q\times 1} \cr  
0_{q\times 1} \cr
\delta_B \cr
\end{pmat}
U_0\\
y &= \begin{pmat}[{.|..}] 0_{1\times q} & \delta_C & 0_{1\times q} & 0_{1\times q} & 0_{1\times q} \cr \end{pmat} x 
\end{aligned}
\label{eq:TL-Blockwise}
\end{equation}
where $I_q$ is the identity matrix of size $q$, $0_{m,n}$ is a zero matrix of size $m\ts\times\ts n$, $I_q^+$ and  $I_q^-$  are squared matrices having ones on the first super- and sub-diagonal respectively and the other entries are defined as follows
\begin{equation*}
	\begin{aligned}
	L_q &\defeq \diag{L'_1,\dots,L'_q}\\
	I_{R} &\defeq \trans{\left[I_{R_1},\dots,I_{R_q}\right]}\\
	\delta_B &\defeq \trans{\left[1, 0, \dots 0\right]}\\
	\delta_C &\defeq \left[0, \dots, 0, 1\right].
	\end{aligned}
\end{equation*}

Note that other techniques like Modified Nodal Analysis (MNA) are widely used to model such networks \cite{Kunkel_2006}. However, they do not directly yield a SE-DAE as in \eqref{eq:SE} and were therefore deemed as less suitable for the purpose of this treatise. It is however possible to transform the resulting equations from said modeling formalisms into a SE-DAE, as the underlying system stays the same.



\section*{Acknowledgements}
\noindent
Ms. Tatjana Stykel is sincerely thanked for the fruitful discussions on differential-algebraic equations.

\bibliographystyle{elsarticle-num} 
\bibliography{Castagnotto_DAE_arXiv}
\end{document}